\documentclass[12pt]{amsart}
\usepackage{amssymb,amsmath,amsthm, amscd, enumerate, mathrsfs}
\usepackage{graphicx, hhline}
\usepackage[all]{xy}
\usepackage{braket}
\usepackage[usenames]{color}
\usepackage{hyperref}
\usepackage{fancyhdr}
\usepackage[top=30truemm,bottom=30truemm,left=30truemm,right=30truemm]{geometry}
\usepackage[pagewise]{lineno}

\hypersetup{colorlinks=true}

\title[Generalized canonical bundle formula]{Addendum: On generalized canonical bundle formula and boundedness of complements in complex analytic setting}
\author{Kenta Hashizume}
\date{2026/08/31}
\keywords{generalized pair, canonical bundle formula}
\subjclass[2020]{Primary: 14D06, Secondary: 14J40, 14C20, 32C15}
\address{Department of 
Mathematics, Faculty of Science, Niigata University, Niigata 950-2181, Japan}
\email{hkenta@math.sc.niigata-u.ac.jp}

\newtheorem{thm}{Theorem}[section]

\newtheorem{prop}[thm]{Proposition}

\theoremstyle{definition}

\newtheorem{rem}[thm]{Remark}

\newtheorem*{ack}{Acknowledgments} 
 
\newtheorem*{b-divisor}{b-divisors}

\newtheorem*{quasi-log}{Quasi-log complex analytic space} 
 
\newtheorem*{g-pair}{Generalized pairs} 
\newtheorem*{adj-g-pair}{Divisorial adjunction for generalized pairs} 
\newtheorem*{mmp-g-pair}{MMP for generalized pairs}

\newtheorem*{claim*}{Claim}

\begin{document}

\begin{abstract}
We establish the generalized canonical bundle formula for generalized lc-trivial fibrations without the assumption on the nef part in the complex analytic setting. 
\end{abstract}

\maketitle

\tableofcontents

\section{Introduction}\label{sec--intro}

In \cite{has-gen-can-bundle-formula}, the author proved the following (see \cite[Theorem 1.1]{has-gen-can-bundle-formula}). 

\begin{thm}\label{thm--gen-can-bundle-formula-intro}
Let $\pi \colon (X,B+M) \to S$ be a generalized sub-pair with the nef part $\boldsymbol{\rm M}$, and let $f\colon (X,B+M) \to Z$ be a generalized lc-trivial fibration over $S$. 
Let $W \subset S$ be a Stein compact subset. 
Suppose that
\begin{itemize}
\item
$\boldsymbol{\rm M}$ is a finite $\mathbb{R}_{>0}$-linear combination of b-nef$/S$ $\mathbb{Q}$-b-Cartier $\mathbb{Q}$-b-divisors on $X$, and
\item
there exists a Zariski open dense subset $U \subset Z$ such that $B|_{f^{-1}(U)}$ is effective.
\end{itemize}
Then, after replacing $S$ with a suitable open neighborhood of $W$, we have the following properties.
\begin{enumerate}[(i)]
\item \label{thm--gen-can-bundle-formula-intro-(i)}
Let $\boldsymbol{\rm G}$ (resp.~$\boldsymbol{\rm N}$) be the discriminant $\mathbb{R}$-b-divisor (resp.~the moduli $\mathbb{R}$-b-divisor) associated to $f$. 
Then $\boldsymbol{\rm N}$ is a finite $\mathbb{R}_{>0}$-linear combination of b-nef$/S$ $\mathbb{Q}$-b-Cartier $\mathbb{Q}$-b-divisors on $Z$. 
In particular, $(Z, \boldsymbol{\rm G}_{Z} +\boldsymbol{\rm N}_{Z}) \to S$ is a generalized sub-pair with the nef part $\boldsymbol{\rm N}$. 
Furthermore, if $(X,B+M) \to S$ is generalized lc (resp.~generalized klt), then $(Z, \boldsymbol{\rm G}_{Z} +\boldsymbol{\rm N}_{Z}) \to S$ is also generalized lc (resp.~generalized klt). 
\item \label{thm--gen-can-bundle-formula-intro-(ii)}
For any open subset $\tilde{S} \subset S$ that does not necessarily contain $W$, if we put $\tilde{X} \subset X$ and $\tilde{Z} \subset Z$ as the inverse images of $\tilde{S}$, then the moduli $\mathbb{R}$-b-divisor $\tilde{\boldsymbol{\rm N}}$ associated to the generalized lc-trivial fibration $(\tilde{X}, B|_{\tilde{X}}+M|_{\tilde{X}}) \to \tilde{Z}$ over $\tilde{S}$ satisfies (\ref{thm--gen-can-bundle-formula-intro-(i)}). 
\end{enumerate}
In particular, we have $\tilde{\boldsymbol{\rm N}}=\boldsymbol{\rm N}|_{\tilde{Z}}$. 
\end{thm}

In this note, we remove the assumption on the nef part $\boldsymbol{\rm M}$ in Theorem \ref{thm--gen-can-bundle-formula-intro}.

\begin{thm}[Main result, complex analytic setting]\label{thm--main-analytic-intro}
Let $\pi \colon (X,B+M) \to S$ be a generalized sub-pair with the nef part $\boldsymbol{\rm M}$, and let $f\colon (X,B+M) \to Z$ be a generalized lc-trivial fibration over $S$. 
Let $W \subset S$ be a Stein compact subset. 
Suppose that there exists a Zariski open dense subset $U \subset Z$ such that $B|_{f^{-1}(U)}$ is effective.
Then, after replacing $S$ with a suitable open neighborhood of $W$, we have the following properties.
\begin{enumerate}[(i)]
\item \label{thm--main-analytic-intro-(i)}
Let $\boldsymbol{\rm G}$ (resp.~$\boldsymbol{\rm N}$) be the discriminant $\mathbb{R}$-b-divisor (resp.~the moduli $\mathbb{R}$-b-divisor) associated to $f$. 
Then $\boldsymbol{\rm N}$ is a b-nef$/S$ $\mathbb{R}$-b-Cartier $\mathbb{R}$-b-divisors on $Z$. 
In particular, $(Z, \boldsymbol{\rm G}_{Z} +\boldsymbol{\rm N}_{Z}) \to S$ is a generalized sub-pair with the nef part $\boldsymbol{\rm N}$. 
Furthermore, if $(X,B+M) \to S$ is generalized lc (resp.~generalized klt), then $(Z, \boldsymbol{\rm G}_{Z} +\boldsymbol{\rm N}_{Z}) \to S$ is also generalized lc (resp.~generalized klt). 
\item \label{thm--main-analytic-intro-(ii)}
For any open subset $\tilde{S} \subset S$ that does not necessarily contain $W$, if we put $\tilde{X} \subset X$ and $\tilde{Z} \subset Z$ as the inverse images of $\tilde{S}$, then the moduli $\mathbb{R}$-b-divisor $\tilde{\boldsymbol{\rm N}}$ associated to the generalized lc-trivial fibration $(\tilde{X}, B|_{\tilde{X}}+M|_{\tilde{X}}) \to \tilde{Z}$ over $\tilde{S}$ satisfies (\ref{thm--main-analytic-intro-(i)}). 
\end{enumerate}
In particular, we have $\tilde{\boldsymbol{\rm N}}=\boldsymbol{\rm N}|_{\tilde{Z}}$. 
Moreover, if $\boldsymbol{\rm M}$ is a finite $\mathbb{R}_{>0}$-linear combination of b-nef$/S$ $\mathbb{Q}$-b-Cartier $\mathbb{Q}$-b-divisors, then so is $\boldsymbol{\rm N}$. 
\end{thm}

The following theorem is the corresponding algebraic setting. 

\begin{thm}[Main result, algebraic setting]\label{thm--main-algebraic-intro}
Let $\pi \colon (X,B+M) \to S$ be a generalized sub-pair over a quasi-projective scheme $S$ with the nef part $\boldsymbol{\rm M}$. 
Let $f\colon (X,B+M) \to Z$ be a generalized lc-trivial fibration over $S$. 
Suppose that $B|_{f^{-1}(U)}$ is effective for some Zariski open subset $U \subset Z$.
Let $\boldsymbol{\rm G}$ (resp.~$\boldsymbol{\rm N}$) be the discriminant $\mathbb{R}$-b-divisor (resp.~the moduli $\mathbb{R}$-b-divisor) associated to $f$. 
Then $\boldsymbol{\rm N}$ is a b-nef$/S$ $\mathbb{R}$-b-Cartier $\mathbb{R}$-b-divisors on $Z$, and $(Z, \boldsymbol{\rm G}_{Z} +\boldsymbol{\rm N}_{Z}) \to S$ is a generalized sub-pair with the nef part $\boldsymbol{\rm N}$. 
Furthermore, if $\boldsymbol{\rm M}$ is a finite $\mathbb{R}_{>0}$-linear combination of b-nef$/S$ $\mathbb{Q}$-b-Cartier $\mathbb{Q}$-b-divisors, then so is $\boldsymbol{\rm N}$. 
If $(X,B+M) \to S$ is generalized lc (resp.~generalized klt), then $(Z, \boldsymbol{\rm G}_{Z} +\boldsymbol{\rm N}_{Z}) \to S$ is also generalized lc (resp.~generalized klt). 
\end{thm}

We note that Theorem \ref{thm--main-algebraic-intro} was proved in \cite[Theorem 11.4.4]{chlx} by using the theory of foliations.

\begin{ack}
The author was partially supported by JSPS KAKENHI Grant Number JP26K06724. 
The work was done when the author attend ``Mini-Workshop on  Complex Geometry related to MMP''. 
The author thanks Professors Atsushi Ito, Takayuki Koike, and Shin-ichi Matsumura, who organized the workshop and invited the author to the workshop.  
\end{ack}

\section{Proofs of main results}\label{sec--proof}

For the reader's convenience, we first treat the algebraic setting (Theorem \ref{thm--main-algebraic-intro}), and then we treat the complex analytic setting (Theorem \ref{thm--main-analytic-intro}). 
We refer the reader to \cite[Definition 2.19]{filipazzi-svaldi} and \cite[Definition 3.1]{has-gen-can-bundle-formula} for the definition of the generalized canonical bundle formula. 
Although \cite[Definition 2.19]{filipazzi-svaldi} is stated in a setting with rational coefficients, the same definition applies to arbitrary generalized pairs (see also \cite[Definition 3.1]{has-gen-can-bundle-formula}).

\begin{prop}\label{prop--M=0}
Let $f\colon (X,B) \to Z$ be an lc-trivial fibration over a quasi-projective scheme $S$. 
Suppose that $B|_{f^{-1}(U)}$ is effective for some Zariski open subset $U \subset Z$.
Then the moduli $\mathbb{R}$-b-divisor $\boldsymbol{\rm N}$  associated to $f$ is a finite $\mathbb{R}_{>0}$-linear combination of b-nef$/S$ $\mathbb{Q}$-b-Cartier $\mathbb{Q}$-b-divisors on $Z$. 
\end{prop}

\begin{proof}
This is a special case of \cite[Theorem 1.7]{fujino-hashizume-adjunction}, or we can easily check that the arguments in \cite[Section 3]{has-gen-can-bundle-formula} work with only minor changes. 
\end{proof}

\begin{proof}[Proof of Theorem \ref{thm--main-algebraic-intro}]
With Proposition \ref{prop--M=0}, we may apply the arguments in \cite{filipazzi-gen-can-bundle-formula} to our situation. 
Indeed, we can use weak semistable reductions \cite{karu-phdthesis}, generalized dlt blow-ups, and the MMP for $\mathbb{Q}$-factorial generalized dlt pairs \cite{bz} even if the nef parts of generalized pairs are not represented by a finite $\mathbb{R}_{>0}$-linear combination of b-nef$/S$ $\mathbb{Q}$-b-Cartier $\mathbb{Q}$-b-divisors.    
\end{proof}

\begin{proof}[Proof of Theorem \ref{thm--main-analytic-intro}]
By \cite[Lemma 2.16]{fujino-analytic-bchm}, we get
$$W \subset W' \subset S,$$
where $W'\subset S$ is a Stein compact subset such that after shrinking $S$ around $W'$, the morphism $\pi \colon X \to S$ and $W'$ satisfy the property (P). 
Then Theorem \ref{thm--main-analytic-intro} for $W$ follows from Theorem \ref{thm--main-analytic-intro} (\ref{thm--main-analytic-intro-(ii)}) for $W'$. 
Hence, we may replace $W$ by $W'$, and therefore we may assume that $\pi$ and $W$ satisfy the property (P). 
Then Theorem \ref{thm--main-analytic-intro} can be proved similarly to the algebraic setting by using the case of $\boldsymbol{\rm M}=0$ in Theorem \ref{thm--gen-can-bundle-formula-intro}, \cite{eh-semistablereduction}, and \cite{fujino-analytic-bchm}.  
\end{proof}

\begin{rem}
We emphasize that all the statements in \cite[Section 3]{has-gen-can-bundle-formula} are necessary to prove Theorem \ref{thm--main-analytic-intro}. 
Indeed, to prove Theorem \ref{thm--main-analytic-intro} we need Theorem \ref{thm--gen-can-bundle-formula-intro} for $\boldsymbol{\rm M}=0$. 
Since we do not use the variations of mixed Hodge structures in \cite{has-gen-can-bundle-formula}, even in the case of lc-trivial fibrations we first need to establish the generalized canonical bundle formula for generalized klt-trivial fibration, and then we need to extend it to the case of generalized lc-trivial fibrations as in Theorem \ref{thm--gen-can-bundle-formula-intro}. 
\end{rem}

Finally, we provide a supplement to define the canonical b-divisors and the moduli $\mathbb{R}$-b-divisors associated to generalized lc-trivial fibrations (\cite[Definition 3.1]{has-gen-can-bundle-formula}). 

\begin{thm}\label{thm--welldef-can-b-div}
Let $\pi \colon X \to S$ be a projective morphism from a normal analytic variety $X$ to a complex analytic space $S$, and let $W \subset S$ be a Stein compact subset. 
Then, after shrinking $S$ around $W$ suitably, we can globally define the canonical b-divisor $\boldsymbol{\rm K}$ on $X$, which is a $\mathbb{Z}$-b-divisor whose traces are the canonical divisors.
In other words, for any proper bimeromorphic morphism $X' \to X$ from a normal analytic variety $X'$, the canonical divisor $K_{X'}$ on $X'$ is well defined as a Weil divisor on $X'$, and furthermore, we can define these canonical divisors so that for any two proper bimeromorphic morphisms $X_{1} \to X$ and $X_{2} \to X$ such that the induced bimeromorphic map $\phi \colon X_{1} \dashrightarrow X_{2}$ is a bimeromorphic contraction, we have $\phi_{*}K_{X_{1}}=K_{X_{2}}$ as Weil divisors on $X_{2}$. 
\end{thm}

\begin{proof}
Let $f \colon Y \to X$ be a resolution of $X$. 
By shrinking $S$ around $W$, we may assume that $S$ is Stein, $\pi \circ f \colon Y \to S$ is projective,   and there exists a $(\pi \circ f)$-very ample divisor $A$ on $Y$ that defines an embedding of $Y$ into $S \times \mathbb{P}^{N}$ for some $N$ (\cite[Lemma 2.2]{eh-semistablereduction}). 
From now on, we do not shrink $S$ anymore. 

Let $\omega_{Y}$ be the canonical bundle on $Y$. 
We fix a positive integer $m$ such that some analytically general fiber $F$ of $\pi \circ f$ satisfies $H^{i}(F,(\omega_{Y}\otimes \mathcal{O}_{Y}(mA))|_{F}) = 0$ for all $i>0$ and $H^{0}(F,(\omega_{Y}\otimes \mathcal{O}_{Y}(mA))|_{F}) \neq 0$.
By the cohomology and base change theorem, we have $(\pi \circ f)_{*}(\omega_{Y}\otimes \mathcal{O}_{Y}(mA))\neq 0$. 
By Cartan's Theorem A (cf.~\cite[Theorem 2.6]{fujino-analytic-bchm}), we see that $(\pi \circ f)_{*}\bigl(\omega_{Y}\otimes \mathcal{O}_{Y}(mA)\bigr)$ is generated by global sections. 
Thus, $\omega_{Y}\otimes \mathcal{O}_{Y}(mA)$ has a non-zero global section $u$. 
Considering the local expression of $u$, we can easily check that $u$ defines a Weil divisor $D$ on $Y$. 
On the other hand, the choice of $A$ shows that $\mathcal{O}_{Y}(mA)$ has a global section $v$. 
Then $K_{Y}:=D-mA$ is the canonical divisor on $Y$ because the multiplication by $\frac{u}{v}$ defines an isomorphism $\mathcal{O}_{Y}(K_{Y})\to \omega_{Y}$. 
Note that $K_{Y}$ depends on the choice of $u$ and $v$. 
We fix $K_{Y}$ with the corresponding meromorphic section $\frac{u}{v}$.

For any proper bimeromorphic morphism $X' \to X$ from a normal analytic variety $X'$, let $Y'$ be the normalization of the main component of $Y \times_{X}X'$ and let $p \colon Y' \to Y$ and $q \colon Y' \to X'$ be the projections. 
Then $p^{*}\bigl(\frac{u}{v}\bigr)$ is a meromorphic section of the canonical sheaf $\omega_{Y'}$, and therefore it globally defines the canonical divisor $K_{Y'}$ on $Y'$ such that $p_{*}K_{Y'}=K_{Y}$. 
Hence, $K_{X'}:=q_{*}K_{Y'}$ is also well defined, and these $K_{X'}$ for all $X' \to X$ form the canonical b-divisor $\boldsymbol{\rm K}$.
\end{proof}


\end{document}